\documentclass[10pt]{amsart}

\usepackage{amsfonts,amssymb,amsmath,amscd,amstext}
\usepackage[colorlinks=true,linkcolor=blue,citecolor=blue]{hyperref}
\usepackage[utf8]{inputenc}
\usepackage{microtype}
\usepackage{graphicx} 
\usepackage{tikz}
\usepackage{changes}

\usepackage{enumitem}

\newcommand{\hh}{{\mathbb{H}}}

\newcommand{\sph}{{\mathbb{S}}}

\newcommand{\ga}{\gamma}

\newcommand{\hhh}{\mathcal{H}}

\newcommand{\dd}{\,\mathrm{d}}	\newcommand{\de}{\partial}		
\renewcommand{\div}{\operatorname{div}}	
	
\newcommand{\R}{\mathbb{R}}	
\newcommand{\N}{\mathbb{N}}	
\newcommand{\HH}{\mathbb{H}}	

\newcommand{\simdiff}{\triangle}
\usepackage{dsfont}
\newcommand{\one}{{\mathds 1\!}} 
\newcommand{\LL}{\mathcal{L}}
\newcommand{\Span}{\operatorname{span}}

\renewcommand{\ge}{\geqslant}
\renewcommand{\le}{\leqslant}

\definechangesauthor[name=Manuel, color=purple]{M} 

\newtheorem{theorem}{Theorem}[section]
\newtheorem{proposition}[theorem]{Proposition}
\newtheorem{lemma}[theorem]{Lemma}
\newtheorem{corollary}[theorem]{Corollary}
\newtheorem{thm}{Theorem}[section]

\theoremstyle{definition}

\theoremstyle{remark}

\numberwithin{equation}{section}

\title{Area-minimizing cones in the Heisenberg group $\hh^1$}

\author[Nicolussi Golo]{Sebastiano Nicolussi Golo}
\address{Department of Mathematics and Statistics, University of Jyväskylä, Jyväskylä, Finland}	
\email{sebastiano2.72@gmail.com}
\thanks{S.N.G.~has been supported
	by University of Padova STARS Project ``Sub-Riemannian Geometry and Geometric Measure Theory Issues: Old and New''.
	Both authors have been supported 
	by the INdAM – GNAMPA Project 2019 ``Rectifiability in Carnot groups''.}

\author{Manuel Ritoré}
\address{Departamento de Geometría y Topología \& Research Unit MNat \\ Universidad de Granada \\ Granada, Spain}
\email{ritore@ugr.es}
\thanks{M. R. has been supported by MEC-Feder grant MTM2017-84851-C2-1-P, Junta de Andalucía grant A-FQM-441-UGR18, MSCA GHAIA, and Research Unit MNat UCE-PP2017-3}

\begin{document}

\date{\today}

\begin{abstract}
We present a characterization of minimal cones of class $C^2$ and $C^1$ in the first Heisenberg group $\mathbb{H}$, with an additional set of examples of minimal cones that are not of class $C^1$.
\end{abstract}

\maketitle

\thispagestyle{empty}

\bibliographystyle{amsplain}
\nocite{*}

\setcounter{tocdepth}{2}
\tableofcontents

\section{Introduction}
The interest towards Geometric Measure Theory in the Heisenberg group
 grew drastically in the last decades, see for instance 
 \cite{MR1404326,MR1871966,MR2983199,2020arXiv200402520J} and the references therein.
Despite many deep results, fundamental questions still remain open,
the main difficulty being that sets of finite perimeter may not be rectifiable sets in the Riemannian sense.

In the effort to understand minimal surfaces in the first Heisenberg group, we are presenting a characterization of minimal cones of class $C^2$ and $C^1$.
Furthermore, we also provide a set of examples of minimal cones that are not of class $C^1$.
By \emph{cone} we mean a set that is invariant under the anisotropic dilations of the Heisenberg group.

Complete minimal surfaces of class $C^2$ have been classified in~\cite{MR2609016}.
We provide a self-contained classification of minimal cones of class $C^2$, as it is a simple exercise in our case.
Minimal surfaces of class $C^1$ have been studied in \cite{MR3412382,MR3406514}.
Tentatives to study minimal surfaces with regularity lower than $C^1$ can be found 
in  \cite{MR3753176,MR3984100}.

The construction of minimal cones is the following, see Section~\ref{sec5ee1fa47} for details. Given proper disjoint open subarcs $I,J$ of the unit circle $\sph^1\subset\R^2$, let $L$ be the bisectrix of $I$. Then consider the family of planar curves made of (see Figure~\ref{fig5ee1f373}):
\begin{enumerate}
\item rays emanating from $0$ and intersecting~$\overline{J}$;
\item the half-line $L$ together with half-lines starting from $L$ parallel to the two boundary lines of $0\# I$ (the cone in $\R^2$ with vertex $0$ over $I$).
\end{enumerate}

\begin{figure}[h]
\label{fig5ee1f373}
\begin{tikzpicture}[scale=1]
\clip(0,0) circle (2.5);
\draw (0,0) circle (2);
\foreach \x in {0,...,14}
\draw[rotate=165,shift={(0.3*\x,0)},rotate=-35,blue] (0,0) -- (3,0);
\foreach \x in {0,...,14}
\draw[rotate=165,shift={(0.3*\x,0)},rotate=35,blue] (0,0) -- (3,0);
\draw[rotate=130, thick,blue] (2,0) arc (0:70:2) node[circle,fill=white,near start] {$I$};
\draw[rotate=130+35,thick,black] (0,0) -- (3,0) node[circle,fill=white,near start]{$L$};
\filldraw [rotate=130,blue] (2,0) circle (0.1);
\filldraw [rotate=200,blue] (2,0) circle (0.1);
\foreach \x in {0,...,12}
\draw[rotate=0+5*\x,purple] (0,0) -- (3,0);
\draw[rotate=0,thick,purple] (2,0) arc (0:60:2) node[circle,fill=white,midway]  {$J$};
\filldraw [purple] (2,0) circle (0.1);
\filldraw [rotate=60,purple] (2,0) circle (0.1);
\end{tikzpicture}
\caption{The configuration of lines in $\R^2$ for given arcs $I, J$.}
\end{figure}
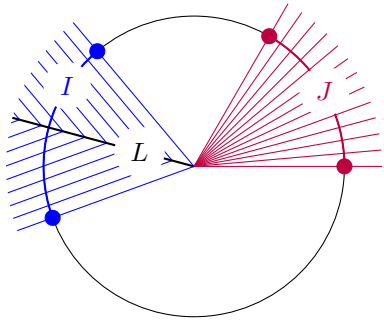

All these curves in $\R^2$ lift uniquely to horizontal curves in $\HH^1$, whose union form a surface $C(I;J)\subset\HH^1$ with non-empty boundary in general. The lifted curves are the characteristic curves of~$C(I;J)$.
We are interested in particular in the conical surface $C(I):=C(I;\sph^1\setminus I)$.

Similarly, we can construct a surface $C(\mathcal I)$ from a (possibly infinite, but however countable) family $\mathcal I$ of disjoint open arcs  of $\sph^1$. Roughly the surface $C(\mathcal{I})$ is built applying the above construction to any connected component $I$ of $\mathcal{I}$ and to any connected component $J$ of $\mathbb{S}^1\setminus \bigcup_{I\in\mathcal{I}} \overline{I}$, see Section~\ref{sec5ee1fa47} and Figure~\ref{fig5ee1f4ec}. These are minimal cones with different degrees of regularity.

\begin{thm}\label{thm5ee1f82e}
	Let $\mathcal I$ be a family of disjoint arcs of $\sph^1$.
	\begin{enumerate}
	\item The surface $C(\mathcal I)$ is a minimal cone.
	\item The surface $C(\mathcal I)$ is of class $C^1$ if and only if $\mathcal I$ is finite and the closure of $\bigcup_{I\in\mathcal I}I$ is $\sph^1$.
	\end{enumerate}
\end{thm}
Theorem~\ref{thm5ee1f82e} is proven in Propositions~\ref{prop5e96b917} and~\ref{prop5ee1f81b}.
With these examples at hand, 
we provide a classification of minimal cones of class $C^1$.
The classification is based on the study of the singular set of minimal surfaces, that is, the set of points where the tangent plane is horizontal, see \cite{MR2983199}.
See Section~\ref{sec5ee1fd1e} for the proof.

\begin{thm}\label{thm5ee1fc47}
	If $S\subset\hh^1$ is a minimal cone of class $C^1$,
	then one of the following possibilities holds:
	\begin{enumerate}
	\item $S$ is a vertical plane, or
	\item $S$ is the horizontal plane $\{t=0\}$, or
	\item $S=C(I_1,\ldots,I_k)$ for some disjoint non-empty open arcs $I_1,\ldots,I_k$ in $\sph^1$ with $\sph^1=\bigcup_{j=1}^k\bar I_j$ and $k\ge2$.
	\end{enumerate}
	These cases can be distinguished by their singular set: empty in the first case, a single point in the second case, and a finite family of horizontal half-lines starting from the vertex in the third case.
\end{thm}

Not all $C^1$ minimal cones are of class $C^2$.
In the third class, the only minimal cones of class $C^2$ are among those with  $k=2$.
\begin{thm}\label{thm5ee202de}
	If $S\subset\hh^1$ is a minimal cone of class $C^2$,
	then $S$ is a vertical plane, or the horizontal plane $\{t=0\}$, or  rotations about the $t$-axis of the graph of the function $t=-xy$.
\end{thm}

Theorem~\ref{thm5ee202de} follows from Theorem 5.1 of \cite{MR2435652}, where it is proven that the unique entire $C^2$ area-stationary graphs over the plane $H$ in $\HH^1$ are Euclidean planes and vertical rotations of graphs of the form $t = xy + (ay + b)$, where $a$ and $b$ are real constants. In case the surface is a cone then $a=b=0$.

\subsection*{Plan of the paper}
The preliminary Section~\ref{sec5ee1fa82} introduces the main definitions and properties of the Heisenberg group that we need.
The construction of minimal cones that we sketched above is presented in detail in Section~\ref{sec5ee1fa47}.
Finally, we prove our main results in Section~\ref{sec5ee200a6}.

\section{Preliminaries}\label{sec5ee1fa82}

\subsection{The Heisenberg group}
We identify the first \emph{Heisenberg group} $\HH$ with $\R^3$ with coordinates $(x,y,t)$ where we set the group operation
\[
(x,y,t) * (x',y',t')
= (x+x' , y+y' , t+t'+ (x'y-xy') ) .
\]
The neutral element is $(0,0,0)$ and the inverse of $(x,y,t)$ is $(-x,-y,-t)$.
The \emph{dilation} of factor $\lambda>0$ centered at $0$ is the Lie group automorphism $\delta_\lambda(x,y,t) = (\lambda x,\lambda y,\lambda^2 t)$.
We choose the frame of left-invariant vector fields generated by $\de_x$, $\de_y$ and $\de_t$ at $0$
\[
X = \de_x + y\, \de_t,
\qquad
Y = \de_y - x\, \de_t,
\qquad
T = \de_t .
\]
Notice that $[X,Y]=-2T$.
These vector fields form a basis for the Lie algebra $\mathfrak h$ of $\HH$, which is stratified with first layer $\hhh=\Span\{X,Y\}$, the \emph{horizontal plane}, and second layer $[\hhh,\hhh]=\Span\{T\}$.

With an abuse of language, we denote by $C^k(\Omega;\hhh)$ (and $C^k_c(\Omega;\hhh)$) the space
of sections of class $C^k$ (with compact support in $\Omega$) of the left-invariant vector bundle generated by $\hhh$.
These sections are vector fields on $\R^3$.

One can easily see that, if $V=v_1X+v_2Y$ with $v_1$ and $v_2$ smooth functions, then the standard divergence in $\R^3$ applied to $V$ is

\[
\div(V) = Xv_1 + Yv_2 .
\]
If we consider the left-invariant Riemannian metric $g$ on $\HH$ making $X,Y,T$ and orthonormal basis, $\div(V)$ is also the divergence with respect to the Riemannian metric $g$.

The left-invariant vector bundle generated by $\hhh$ is the kernel of the contact form
\[
\omega = \dd t - y\dd x + x \dd y .
\]
Lipschitz curves in $\R^2$ can be lifted to $\HH$ in the following way.
\begin{lemma}\label{lem5e9758f5}
	Let $\gamma:[0,1]\to\R^2$, $\gamma(s)=(x(s),y(s))$, be a Lipschitz curve with $\gamma(0)=0$.
	Define $t:[0,1]\to\R$ by
	\[
	t(s) = \int_0^s  (y\dd x-x\dd y)[\gamma'(u)] \dd u
	= \int_0^s (y(u)x'(u) - x(u)y'(u)) \dd u .
	\]
	Then, the curve $s\mapsto(x(s),y(s),t(s))$ is the only horizontal Lipschitz curve in $\HH$ starting from $(0,0,0)$ and projecting to $\gamma$.
	
	Moreover, if $A(\gamma,s)=\{v\gamma(u):u\in[0,s],\ v\in[0,1]\}$ (with the orientation given by $\gamma$), then 
	\[
	t(s) = -2 \int_{A(\gamma,s)} \dd x\wedge\dd y ,
	\]
	which is called the \emph{balayage area spanned by $\gamma$}.
\end{lemma}
\begin{proof}
	Notice that a Lipschitz curve $\eta:s\mapsto(x(s),y(s),t(s))$ is horizontal if and only if $\omega|_{\eta(s)}[\eta'(s)] = 0$ for almost all $s$,
	that is, $t'=yx'-xy'$.
	Integrating, we get the statement.
\end{proof}

\subsection{Sub-Riemannian perimeter}\label{sec5ee1facd}
Given a measurable set $E\subset\HH$ and an open set $\Omega\subset\HH$, the \emph{perimeter} of $E$ in $\Omega$ is defined as
\[
P(E;\Omega) := \sup\left\{ \int_E \div(V) \dd\LL^3 : V\in C^1_c(\Omega;\hhh),\ |V|\le1\right\} ,
\]
where $\LL^3$ is the Lebesgue measure in $\R^3$ that is, in our chosen coordinate system, a Haar measure of $\HH$.

A measurable set $E\subset\HH$ has \emph{locally finite perimeter} if for every bounded open set $\Omega\subset\HH$ we have $P(E;\Omega)<\infty$.
It turns out (see \cite{MR1871966}) that for a locally finite perimeter, the distributional gradient of the characteristic function $\one_E$ is a vector valued Radon measure, that is, there is a positive Radon measure $|\de E|$ and a unit horizontal vector field $\nu_E:\HH\to H$ such that $\nabla\one_E = \nu_E |\de E|$.
The measure $|\de E|$, and thus $\nabla\one_E$, is supported on the so-called \emph{reduced boundary} $\de^*E\subset\de E$.

\begin{proposition}[{\cite{MR1871966}}]
	Let $E\subset\HH^1$ be a set with locally finite perimeter and $V\in C^\infty(\HH;H)$ a smooth horizontal vector field, then
	\begin{equation}\label{eq5e956226}
	\int_E \div(V) \dd\LL^3  = - \int_{\de^*E} \langle V,\nu_E \rangle \dd|\de E| .
	\end{equation}
\end{proposition}

As a corollary, we can easily prove the following formula.

\begin{corollary}
	Let $V\in C^1(\Omega;\hhh)$, $\phi\in C^1(\HH)$ and $E\subset\HH$ a set with locally finite perimeter.
	Then
	\begin{equation}\label{eq5e95675e}
	\int_E \langle \nabla\phi,V \rangle \dd\LL^3 
	= - \int_{\de^*E}\phi \langle V,\nu_E \rangle \dd|\de E| - \int_E\phi \div(V)\dd\LL^3 .
	\end{equation}
\end{corollary}
\begin{proof}
	First, by group convolution, the relation~\eqref{eq5e956226} remains true for $V$ of class $C^1$.
	Second, notice that $\div(\phi V) = \langle \nabla\phi,V \rangle + \phi \div(V)$.
	Therefore, on the one hand,
	\[
	\int_E \langle \nabla\phi,V \rangle \dd\LL^3
	= \int_E \div(\phi V) - \int_E \phi \div(V) \dd\LL^3 ,
	\]
	on the other hand,
	\[
	\int_E \div(\phi V) = - \int_{\de^*E}\phi \langle V,\nu_E \rangle \dd|\de E| .
	\]
by \eqref{eq5e956226}. 
	Putting these two identities together, we get~\eqref{eq5e95675e}.
\end{proof}

We are interested in perimeter minimizers.
A measurable set $E\subset\HH$ is a \emph{perimeter minimizer} in an open set $\Omega\subset\HH$ if, for every $F\subset\HH$ of locally finite perimeter with $E\simdiff F\Subset\Omega$, we have
\[
P(E;\Omega) \le P(F;\Omega) .
\]

A set is a \emph{local perimeter minimizer} if it is a perimeter minimizer in every bounded open set.
A surface $S$ in $\HH$ is an \emph{area-minimizing surface}, or just a \emph{minimal surface}, if it coincides with the reduced boundary of a perimeter minimizer.
The following proposition yields a method via calibrations to prove that a given set is perimeter minimizer.
 
\begin{proposition}[{\cite[Theorem 2.1]{MR2455341}}]\label{prop5e95c6fe}
	Let $E\subset\HH$ be a measurable set, $\Omega\subset\HH$ an open set and $v:\Omega\to H$ a Borel map.
	Assume that
	\begin{enumerate}[label=\emph{(\roman*)}]
	\item	$E$ has locally finite perimeter in $\Omega$;
	\item $v=\nu_E$ $|\de E|$-almost everywhere in $\Omega$;
	\item there exists an open set $\tilde\Omega\subset\Omega$ such that $|\de E|(\Omega\setminus\tilde\Omega)=0$ and $v$ is continuous on $\tilde\Omega$;
	\item $\div(v)=0$ in distributional sense in $\Omega$. 
	\end{enumerate}
	Then $E$ is a perimeter minimizer in $\Omega$.
\end{proposition}

The vector field $v$ above is called a \emph{calibration} for $\de^*E$.
In applications of Proposition~\ref{prop5e95c6fe}, we will give the calibration $v$ by putting together smooth vector fields in different domains.
The following proposition gives a way to check that the resulting vector field has zero distributional divergence.
Notice that condition~\eqref{eq5e956c15} below is automatically satisfied if $v$ is continuous.

\begin{proposition}\label{prop5e95c720}
	Let $\{\Omega_j\}_j$ be a family of open disjoint sets with locally finite perimeter in $\HH$ such that  
	$\{\bar\Omega_j\}_j$ is a locally finite cover of $\HH$ with $\LL^3(\HH\setminus\bigcup\Omega_j) = 0$. 
	For each $j$, let $V_j\in C^1(\bar\Omega_j;\hhh)$ be a horizontal vector field of class $C^1$ on $\bar\Omega_j$ $($extensible to a $C^1$ horizontal vector field on a neighborhood of $\bar\Omega_j)$.

	The distributional divergence of $V:=\sum_j V_j\one_{\Omega_j}$ is zero if and only if $\div(V_j|_{\Omega_j})=0$ for every $j$ and
	\begin{equation}\label{eq5e956c15}
	\sum_j \langle V_j(p),\nu_{\Omega_j}(p) \rangle = 0
	\qquad\text{for $\sum_j|\de\Omega_j|$-a.e.~$p\in\HH$,}
	\end{equation}
	where we set $\nu_{\Omega_j}(p)=0$ if $p\notin\de^*\Omega_j$, so that the first series is a finite sum for every~$p$.
\end{proposition}
\begin{proof}
	Let $\phi\in C^\infty_c(\HH)$.
	Using~\eqref{eq5e95675e}, we have
%	\begin{align*}
%	\int_\HH \langle \nabla\phi,V \rangle \dd\LL^3
%	&= \sum_j \int_{\Omega_j} \langle \nabla\phi,V_j \rangle \dd\LL^3 \\
%	&= -\sum_j \left( \int_{\de^*\Omega_j} \phi \langle V_j,\nu_{\Omega_j} \rangle \dd|\de \Omega_j| + \int_{\Omega_j} \phi \div(V_j)\dd\LL^3 \right) .
%	\end{align*}
	\begin{equation}\label{eq5ff98651}
	\begin{aligned}
	\int_\HH \langle \nabla\phi,V \rangle \dd\LL^3
	&= \sum_j \int_{\Omega_j} \langle \nabla\phi,V_j \rangle \dd\LL^3 \\
	&= -\sum_j \left( \int_{\de^*\Omega_j} \phi \langle V_j,\nu_{\Omega_j} \rangle \dd|\de \Omega_j| + \int_{\Omega_j} \phi \div(V_j)\dd\LL^3 \right) .
	\end{aligned}
	\end{equation}
%	The latter expression is zero for every $\phi\in C^\infty_c(\HH)$ if and only if  $\div(V_j|_{\Omega_j})=0$ for every $j$ and~\eqref{eq5e956c15} holds.
	On the one hand, if $\div(V_j|_{\Omega_j})=0$ for every $j$ and~\eqref{eq5e956c15} holds, then the integral in~\eqref{eq5ff98651} is zero for every $\phi\in C^\infty_c(\HH)$.
	On the other hand, suppose the expression in~\eqref{eq5ff98651} is zero for every $\phi\in C^\infty_c(\HH)$.
	Then, we have in particular $\int_{\Omega_j} \phi \div(V_j)\dd\LL^3 = 0$ for every $\phi\in C^\infty_c(\Omega_j)$ and every $j$; 
	hence, $\div(V_j)\equiv0$ for every $j$.
	What remains in~\eqref{eq5ff98651} is then~\eqref{eq5e956c15}. 
\end{proof}

Finally, the following stability of perimeter minimizers is well known.

\begin{proposition}\label{prop5e96e7aa}
	Let $\{E_k\}_{k\in\N}$ be a sequence of local perimeter minimizers and $E$ a set of locally finite perimeter such that $\one_{E_k}$ converge locally in $L^1$ to $\one_E$.
	Then $E$ is also local perimeter minimizer.
\end{proposition}

\subsection{Regularity of $C^1$ area-minimizing surfaces in $\hh^1$}
Given a $C^1$ surface $S$, the set $S_0\subset S$ is composed of the points $p$ where $T_pS$ is horizontal. It is referred to as the \emph{singular set} of $S$.
Points in $S\setminus S_0$ are called \emph{regular points}.
A \emph{horizontal line segment} is the image in $\HH$ of an interval in $\R$ through a curve of the form $s\mapsto p\exp(sv)$, for some $p\in\HH$ and $v\in H$.

\begin{proposition}[{\cite{MR2983199,MR3406514,MR3412382}}]\label{prop5e97045a}
	If $S$ is a minimal $C^1$ surface, then $S\setminus S_0$ is ruled by horizontal line segments whose endpoints lie in $S_0$. If $S$ is a $t$-graph then at most one endpoint lies in $S_0$.

\end{proposition}

Given a function $u:A\to\R$ defined on a domain $A\subset\R^2$, its \emph{$t$-graph} is the surface $\{(x,y,u(x,y)):(x,y)\in A\}$.
We always consider a $t$-graph as boundary of the subgraph $E:=\{(x,y,t):t\le u(x,y),\ (x,y)\in A\}$.
The following lemma characterize minimal $t$-graphs of continuous functions.

\begin{lemma}\label{lem5e973be0}
	The $t$-graph $S$ of a continuous function $u:\R^2\to\R$ is a minimal surface if and only if the unit normal of $S$, extended to $\HH$ as a Borel vector field independent of $t$, has zero distributional divergence.
\end{lemma}

\section{Construction of minimal cones}\label{sec5ee1fa47}

Consider a finite family $I_1,\ldots,I_k$ of disjoint open arcs in the unit circle $\sph^1$, and let $J_1,\ldots,J_r$ be the open connected arcs in $\sph^1\setminus\bigcup_{i=1}^k \overline{I_i}$. This set could be empty if $\sph^1=\bigcup_{i=1}^k \overline{I_i}$. Let $2\alpha_i$ be the length (opening angle) of $I_i$ and $L_i$ be the bisectrix of the arc $I_i$. 

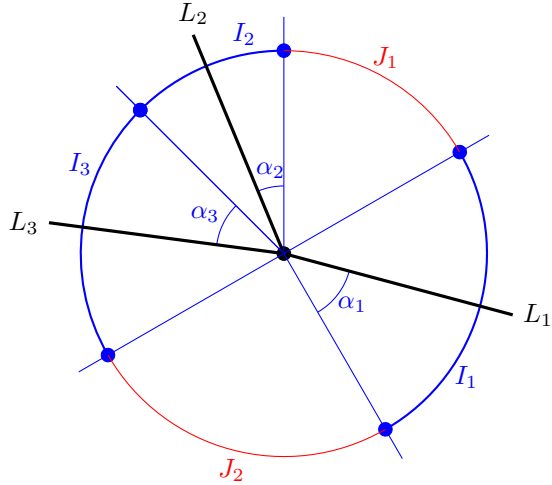
\begin{figure}[h]

\begin{tikzpicture}[scale=0.9]

\filldraw [black] (0,0) circle (0.1);
\draw[rotate=90, thick,blue] (3,0) arc (0:45:3) node[near start, above] {$I_2$};
\draw[rotate=90,blue] (1,0) arc (0:22.5:1) node[midway,above] {$\alpha_2$};
\filldraw [rotate=90,blue] (3,0) circle (0.1);
\filldraw [rotate=135,blue] (3,0) circle (0.1);
\draw[very thick,rotate=112.5] (0,0) -- (3.5,0) node[above] {$L_2$};
\draw[rotate=90,blue] (0,0) -- (3.5,0);
\draw[rotate=135,blue] (0,0) -- (3.5,0);
\draw[rotate=135, thick,blue] (3,0) arc (0:75:3) node[near start, left] {$I_3$};
\draw[rotate=135,blue] (1,0) arc (0:37.5:1) node[near start,left] {$\alpha_3$};
\filldraw [rotate=135,blue] (3,0) circle (0.1);
\filldraw [rotate=210,blue] (3,0) circle (0.1);
\draw[very thick,rotate=172.5] (0,0) -- (3.5,0) node[left] {$L_3$};
\draw[rotate=135,blue] (0,0) -- (3.5,0);
\draw[rotate=210,blue] (0,0) -- (3.5,0);
\draw[rotate=-60, thick,blue] (3,0) arc (0:90:3) node[near start, right] {$I_1$};
\draw[rotate=-60,blue] (1,0) arc (0:45:1) node[near start,right] {$\alpha_1$};
\filldraw [rotate=30,blue] (3,0) circle (0.1);
\filldraw [rotate=-60,blue] (3,0) circle (0.1);
\draw[very thick,rotate=-15] (0,0) -- (3.5,0) node[right] {$L_1$};
\draw[rotate=30,blue] (0,0) -- (3.5,0);
\draw[rotate=-60,blue] (0,0) -- (3.5,0);
\draw[rotate=30,red] (3,0) arc (0:60:3) node[above,midway]  {$J_1$};
\draw[rotate=210,red] (3,0) arc (0:90:3) node[midway,below] {$J_2$};
\end{tikzpicture}

\caption{An initial configuration with three open arcs  $I_1$, $I_2$, $I_3$ }
\end{figure}

The conical sector $0\# \overline{J_i}$ in $\R^2$ (the cone of vertex 0 over $\overline{J_i}$) is filled with half-lines leaving the origin. The conical sectors $0\#\overline{I_i}$ are filled with pairs of half-lines making angle $\alpha_i$ with the half-line $L_i$. 
This way, every point of $\R^2$ can be joined to some $L_i$ or $0$ by a unique shortest path that follows these lines. We lift these paths to $\HH$ as in Lemma~\ref{lem5e9758f5}.
So, we first lift as horizontal curves the half-lines $L_1,\ldots,L_k$ and those in the sectors $0\#\overline{J_i}$, $i=1,\ldots,r$, which remain in the plane $\{t=0\}$. 
Then, for $i=1,\ldots,k$, we lift the half-lines making angle $\alpha_i$ with $L_i$ to horizontal half-lines starting from the corresponding lifted line $L_i$.

We obtain a surface, which we call $C(I_1,\ldots,I_k)$, that is the $t$-graph of a function $u:\R^2\to\R$.
The following lemma gives an explicit formula in a specific case.
Notice that, up to a rotation of $\R^2$, the restriction of $u$ to $0\#I_i$ is equal to the function $u_{\alpha_i}$ on $0\#I$ described below.

\begin{lemma}[{\cite{MR2448649}}]
\label{lem:3.1}
	Let $\alpha\in(0,\pi)$ and define the open arc $I=\{(\cos(\theta),\sin(\theta)):|\theta|<\alpha\}\subset\sph^1$.
	Then $C(I)$ is the $t$-graph of the function
	\begin{equation}\label{eq5e963809}
	u_\alpha(x,y) = 
	\begin{cases}
	y (|y|\cot\alpha-x) & \text{if }(x,y) \in 0\# I , \\
	0 & \text{otherwise}.
	\end{cases}
	\end{equation}
	The function $u_\alpha$ is continuous, but not $C^1$, and has derivatives
	\begin{align*}
		\de_xu_\alpha(x,y) &= 
		\begin{cases}
		-y & \text{if }(x,y) \in 0\# I , \\
		0 & \text{if }(x,y) \in \R^2\setminus\overline{0\# I} ;
		\end{cases} \\
		\de_yu_\alpha(x,y) &= 
		\begin{cases}
		2|y|\cot\alpha-x & \text{if }(x,y) \in 0\# I , \\
		0 & \text{if }(x,y) \in \R^2\setminus\overline{0\# I} .
		\end{cases} \\
	\end{align*}
\end{lemma}
\begin{proof}
	The value of the function $u_\alpha:\R^2\to\R$ at a point $(x,y)$ is the balayage area of the curve from $(0,0)$ to $(x,y)$ that follows the half-lines singled out in the above construction.
	So, if $(x,y)\notin 0\#I$, then $u_\alpha(x,y)=0$.
	If $(x,y)\in0\#I$ and $y\ge0$, then there are $x_0\ge0$ and $s\ge0$ such that
	\[
	\begin{cases}
	x=x_0 + s \cos(\alpha), \\
	y=s\sin(\alpha), 
	\end{cases}
\]
that is
\[
%	\quad\text{that is}\quad
	\begin{cases}
	x_0 = x-y\cot(\alpha), \\
	s= \frac{y}{\sin(\alpha)}.
	\end{cases}
	\]
	So define $u_\alpha(x,y)$ as the Balayage area of the curve from $(0,0)$ to $(x,y)$ that follows the $x$-axis until $(x_0,0)$ and then follows the line parallel to $(\cos\alpha,\sin\alpha)$, that is,
	\[
	u_\alpha(x,y) = 
	-2 \frac{x_0 s\sin(\alpha)}{2} = 
	y (y\cot\alpha-x) .
	\]
	Similarly, if $(x,y)\in0\#I$ and $y\le0$, one finds that 
	$u_\alpha(x,y) = y (-y\cot\alpha-x)$ and so~\eqref{eq5e963809} is proven.
\end{proof}

\begin{proposition}
Let $I_1,\ldots,I_k$ be a finite set of disjoint open arcs in $\sph^1$ and $C(I_1,\ldots,I_k)$ the associated surface. Then
\begin{enumerate}
\item $C(I_1,\ldots,I_k)$ is a conical continuous $t$-graph with vertex at $0$;
\item $C(I_1,\ldots,I_k)$ is a $C^{1,1}$ surface outside the lines $0\#\de J_i$, with singular set $\bigcup_{i=1}^k L_i$. 
It is not $C^2$ at points of the singular set  unless $\alpha=\pi/2$.
\item  $C(I_1,\ldots,I_k)$ is area-minimizing.
\item  The horizontal unit normal of $C(I_1,\ldots,I_k)$ is continuous.
\end{enumerate}
\end{proposition}

\begin{proof}
	The surface $C(I_1,\ldots,I_k)$ is the graph of version of the function \eqref{eq5e963809} in each sector $0\# J_k$, up to a pre-composition with a rotation of the plane.
	Therefore, the first two statements are clear.
	
	We prove that $C(I_1,\ldots,I_k)$ is area-minimizing by presenting a calibration and thus applying Proposition~\ref{prop5e95c6fe}.
	Figure~\ref{fig5e976fc0} helps the understanding.
	Let $v$ be the horizontal vector field that is invariant along $t$ and that is equal to the upward unit normal to $C(I_1,\ldots,I_k)$ outside the half-lines $\bigcup_{i=1}^k L_i$ and $\bigcup_{i=1}^k 0\#\de I_j$.
	We claim that the distributional divergence of $v$ is zero.

	In fact, the unit normal of $C(I_1,\ldots,I_k)$ is the upward unit horizontal vector that is orthogonal to the horizontal characteristic lines we lifted.
	Above the sectors $0\# J_j$ is simply $(y X-xY)/(x^2+y^2)^{1/2}$, which is actually the calibration of the plane $\{t=0\}$;
	in particular, inside the interior of such regions, it is smooth and with zero divergence.
	Above the other sectors, $v$ has constant coefficients in the basis $(X,Y)$ above the regions between the half-lines $L_j$ and the boundaries $\#\de I_j$,
	where it has thus zero divergence.
	
	Finally, one easily sees that $v$ satisfies~\eqref{eq5e956c15} above the half-lines $\#\de I_j$ and the lines $L_j$.

	We conclude that $\div(v)=0$ by Proposition~\ref{prop5e95c720}.
\end{proof}

\begin{proposition}\label{prop5e96b917}
	Given a finite set of disjoint open arcs $I_1,\ldots,I_k$  in $\sph^1$, 
	the associated surface $C(I_1,\ldots,I_k)$ is of class $C^1$ if and only if $\sph^1=\bigcup_{j=1}^k\bar I_j$.
\end{proposition}
\begin{proof}
	Let $u:\R^2\to\R$ be the function whose $t$-graph is $C(I_1,\ldots,I_k)$.
	In each sector $0\#I_j$, the function $u$ is a version of $u_{\alpha_j}$ as in~\eqref{eq5e963809}, up to a rotation of the plane. 

	Since $\de_xu_\alpha$ is not continuous along the half-lines $0\#\de I$, then we conclude that, if $C(I_1,\ldots,I_k)$ is of class $C^1$, then $\sph^1=\bigcup_{j=1}^k\bar I_j$.
	
	Next, notice that the derivative of $u_\alpha$ along the vector $(\cos\alpha,\sin\alpha)$ (or the vector $(\cos(\alpha),-\sin(\alpha))$) is continuous in the half-plane $\{y>0\}$ (in the half-plane $\{y<0\}$, respectively), and zero along the half-line $0\#(\cos\alpha,\sin\alpha)$ (or $0\#(\cos\alpha,-\sin\alpha)$, respectively).
	
	So, if $0\#\bar I_1$ and $0\#\bar I_2$ share a half-line $0\#\hat v$, where $|\hat v|=1$, then the derivative of $u$ along $\hat v$ is continuous across $0\#\hat v$.
	
	What remains to be checked is the continuity across $0\#\hat v$ of the derivative of $u$ along the orthogonal direction to $\hat v$.
	Going back to $u_\alpha$, a computation shows that
	\[
	(-\sin(\alpha) \de_xu_\alpha + \cos(\alpha)\de_yu_\alpha)|_{(s\cos\alpha,s\sin\alpha)}
	= s 
	\]
	and
	\[
	(\sin(\alpha) \de_xu_\alpha + \cos(\alpha)\de_yu_\alpha)|_{(s\cos\alpha,-s\sin\alpha)}
	= s  ,
	\]
	where the derivatives are the continuous limit from inside $0\# I$.
	Therefore, the derivatives of $u$ along the orthogonal direction to $\hat v$ are continuous across $0\#\hat v$.
\end{proof}

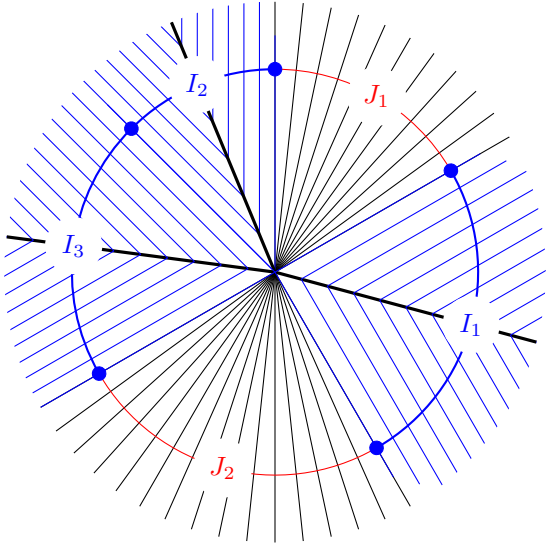
\begin{figure}[h]\label{fig5ee1f4ec}

\begin{tikzpicture}[scale=0.9]
\clip (0,0) circle (4);

\foreach \x in {0,...,10}
\draw[rotate=30+6*\x] (0,0) -- (4,0);
\draw[rotate=30,red] (3,0) arc (0:60:3) node[midway,circle,fill=white]  {$J_1$};

\foreach \x in {0,...,15}
\draw[rotate=210+(90/15)*\x] (0,0) -- (4,0);
\foreach \x in {0,...,10}
\draw[rotate=112.5,shift={(0.6*\x,0)},rotate=22.5,blue] (0,0) -- (3.5,0);
\foreach \x in {0,...,10}
\draw[rotate=112.5,shift={(0.6*\x,0)},rotate=-22.5,blue] (0,0) -- (3.5,0);
\draw[very thick,rotate=112.5] (0,0) -- (4,0) node[above] {$L_2$};
\draw[rotate=90, thick,blue] (3,0) arc (0:45:3) node[midway,circle,fill=white] {$I_2$};
\filldraw [rotate=90,blue] (3,0) circle (0.1);
\draw[rotate=210,red] (3,0) arc (0:90:3) node[midway,circle,fill=white] {$J_2$};

\filldraw [rotate=135,blue] (3,0) circle (0.1);
\filldraw [rotate=210,blue] (3,0) circle (0.1);
\draw[very thick,rotate=172.5] (0,0) -- (4,0) node[left] {$L_3$};
\foreach \x in {0,...,10}
\draw[rotate=135+37.5,shift={(0.4*\x,0)},rotate=37.5,blue] (0,0) -- (4,0);
\foreach \x in {0,...,10}
\draw[rotate=135+37.5,shift={(0.4*\x,0)},rotate=-37.5,blue] (0,0) -- (4,0);
\draw[rotate=135, thick,blue] (3,0) arc (0:75:3) node[midway,circle,fill=white] {$I_3$};

\draw[very thick,rotate=-15] (0,0) -- (4,0) node[right] {$L_1$};
\filldraw [rotate=30,blue] (3,0) circle (0.1);
\filldraw [rotate=-60,blue] (3,0) circle (0.1);
\foreach \x in {0,...,10}
\draw[rotate=-15,shift={(0.4*\x,0)},rotate=45,blue] (0,0) -- (3.5,0);
\foreach \x in {0,...,10}
\draw[rotate=-15,shift={(0.4*\x,0)},rotate=-45,blue] (0,0) -- (3.5,0);
\draw[rotate=-60, thick,blue] (3,0) arc (0:90:3) node[midway,circle,fill=white] {$I_1$};

\end{tikzpicture}

\caption{The complete configuration}
\end{figure}

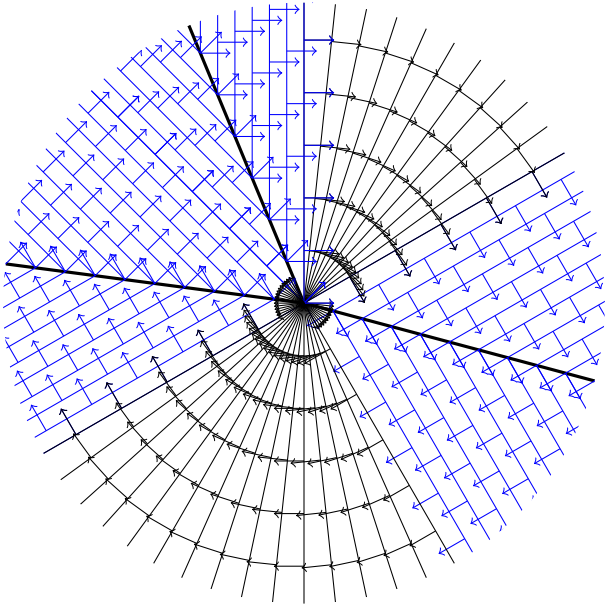
\begin{figure}[h]\label{fig5e976fc0}

\begin{tikzpicture}[scale=1]\jobname{fig5e977402}
\clip (0,0) circle (4);

\newcommand{\vleng}{0.4}
\newcommand{\vstep}{0.7}

\draw[very thick,rotate=-15] (0,0) -- (4,0) ; 
\foreach \x in {0,...,10} {
	\draw[rotate=-15,shift={(0.4*\x,0)},rotate=45,blue] (0,0) -- (4,0) ;
	\foreach \vx in {0,...,10}
		\draw[->,rotate=-15,shift={(0.4*\x,0)},rotate=45,blue] (\vstep*\vx,0) -- (\vstep*\vx, -\vleng);
}
\foreach \x in {1,...,10} {
	\draw[rotate=-15,shift={(0.4*\x,0)},rotate=-45,blue] (0,0) -- (4,0);
	\foreach \vx in {0,...,10}
		\draw[->,rotate=-15,shift={(0.4*\x,0)},rotate=-45,blue] (\vstep*\vx,0) -- (\vstep*\vx, -\vleng);
}

\foreach \x in {0,...,10} {
	\draw[rotate=30+6*\x] (0,0) -- (4,0);
	\foreach \vx in {0,...,10}
		\draw[->,rotate=30+6*\x] (\vstep*\vx,0) -- (\vstep*\vx, -\vleng);
}

\draw[very thick,rotate=112.5] (0,0) -- (4,0) ; 
\foreach \x in {0,...,10} {
	\draw[rotate=112.5,shift={(0.6*\x,0)},rotate=-22.5,blue] (0,0) -- (4,0);
	\foreach \vx in {0,...,10}
		\draw[->,rotate=112.5,shift={(0.6*\x,0)},rotate=-22.5,blue] (\vstep*\vx,0) -- (\vstep*\vx, -\vleng);
}
\foreach \x in {0,...,10} {
	\draw[rotate=112.5,shift={(0.6*\x,0)},rotate=22.5,blue] (0,0) -- (4,0);
	\foreach \vx in {0,...,10}
		\draw[->,rotate=112.5,shift={(0.6*\x,0)},rotate=22.5,blue] (\vstep*\vx,0) -- (\vstep*\vx, -\vleng);
}

\draw[very thick,rotate=172.5] (0,0) -- (4,0) ; 
\foreach \x in {0,...,10} {
	\draw[rotate=135+37.5,shift={(0.4*\x,0)},rotate=-37.5,blue] (0,0) -- (4,0);
	\foreach \vx in {0,...,10}
		\draw[->,rotate=135+37.5,shift={(0.4*\x,0)},rotate=-37.5,blue] (\vstep*\vx,0) -- (\vstep*\vx, -\vleng);
}
\foreach \x in {0,...,10} {
	\draw[rotate=135+37.5,shift={(0.4*\x,0)},rotate=37.5,blue] (0,0) -- (4,0);
	\foreach \vx in {0,...,10}
		\draw[->,rotate=135+37.5,shift={(0.4*\x,0)},rotate=37.5,blue] (\vstep*\vx,0) -- (\vstep*\vx, -\vleng);
}

\foreach \x in {0,...,15} {
	\draw[rotate=210+(90/15)*\x] (0,0) -- (4,0);
	\foreach \vx in {0,...,10}
		\draw[->,rotate=210+(90/15)*\x] (\vstep*\vx,0) -- (\vstep*\vx, -\vleng);
}

\end{tikzpicture}

\caption{The calibration of $C(I_1,\ldots,I_k)$}
\end{figure}

In the special case of two disjoint open intervals $I_1, I_2$ such that $\sph^1=\overline{I_1}\cup\overline{I_2}$, the singular line is a horizontal straight line $L$ whose complement is foliated by two families of parallel lines making a constant angle with $L$. This is merely $C^{1,1}$ except in the case $\alpha=\pi/2$ when we get the cone $t\le -xy$ with $C^\infty$ boundary.

Via approximation, we can consider also the above cones constructed using infinitely many arcs.
More precisely, let $\mathcal I$ be a family of disjoint open arcs of $\sph^1$, possibly countable.
For each $I\in\mathcal I$, let $u_I$ be the function whose $t$-graph is $C(I)$.
Define
\begin{equation}\label{eq5e96ca70}
u_{\mathcal I} = \sum_{I\in \mathcal I} u_I ,
\end{equation}
where the sum is well defined, because for every $v\in\R^2$ there is at most one $I\in\mathcal I$ with $u_I(v)\neq0$.

\begin{proposition}\label{prop5ee1f81b}
	Given a family $\mathcal I$ of disjoint open arcs of $\sph^1$, the function $u_{\mathcal I}$ is continuous and its $t$-graph $C(\mathcal I)$ is a minimal cone.
	Moreover, if $\mathcal I$ is infinite, then $C(\mathcal I)$ is not a $C^1$ surface.
\end{proposition}
\begin{proof}
	From~\eqref{eq5e963809}, one easily sees that $|u_\alpha(v)|\le |v|^2\tan(\alpha)$.
	We deduce that the sum in~\eqref{eq5e96ca70} converges uniformly on compact sets.
	So, $u_{\mathcal I}$ is continuous and its $t$-graph is a cone.
	By Proposition~\ref{prop5e96e7aa}, $C(\mathcal I)$ is a minimal surface.
	
	Finally, if $\mathcal I$ is infinite, then there are $\hat v\in\sph^1$ and a sequence $\{I_k\}_k\subset\mathcal I$ so that $dist(\hat v,I_k)\to0$ and the amplitude of $I_k$ also goes to zero.
	Now, if we consider the function $u_\alpha$ in \eqref{eq5e963809}, we see that its $y$-derivative is
	\[
	\de_yu_\alpha(x,y) = 
	\begin{cases}
	2|y|\cot\alpha-x & \text{if }(x,y) \in 0\#I , \\
	0 & \text{if }(x,y) \in \R^2\setminus\overline{0\#I} .
	\end{cases}
	\]
	In particular, if $(x,y) \in 0\#I $ is close enough to $(1,\tan(\alpha))$, then $\de_yu_\alpha(x,y)$ is arbitrary close to $1$, while $\de_yu_\alpha(1,0)=-1$.
	We conclude that for every $k$ there are points in $0\#I_k$ where some derivative of $u_{\mathcal I}$ oscillates between $1$ and $-1$, so $\nabla u_{\mathcal I}$ is not continuous at $\hat v$.
	Since $\nabla u_{\mathcal I}$ remains bounded, $C(\mathcal I)$ is not a $C^1$ surface.
\end{proof}

\section{Classification results}\label{sec5ee200a6}

\subsection{Characterization of $C^1$ minimal cones}\label{sec5ee1fd1e}

This section is devoted to the proof of our main classification result in the $C^1$ case, 
Theorem~\ref{thm5ee1fc47}.

\begin{lemma}
\label{lem:plane}
A conical $C^1$ surface $S\subset\HH^1$ without singular points is a vertical plane.

\end{lemma}

\begin{proof}
For any $p=(p_1,p_2,p_3)$ in $S$ out of the vertical axis $V$ we consider the curve $\ga(s)=(sp_1,sp_2,s^2p_3)$, whose tangent vector at $s=0$ is the horizontal vector $\ga'(0)=p_1X_0+p_2Y_0\neq 0$. Since $0$ is not a singular point, $S\setminus V$ must be contained in the vertical plane $p_2x-p_1y=0$ and so is a vertical plane.
\end{proof}

\begin{lemma}
\label{lem:half-line}
Let $S\subset \hh^1$ be a conical $C^1$ surface, and let $p\in S_0\setminus\{0\}$. Then $0$ and $p$ belong to a horizontal half-line contained in $S_0$.
\end{lemma}

\begin{proof}
We let $p=(p_1,p_2,p_3)$ and consider the curve $\ga(s):=(sp_1,sp_2,s^2p_3)$, whose image is contained in $S$. We trivially  have $\ga'(s)=p_1 X_{\ga(s)}+p_2Y_{\ga(s)}+2tp_3 T_{\ga(s)}$. Since $\ga(1)=p$ and $p$ is a singular point, the vector $\ga'(1)$ is horizontal and so $p_3=0$. This implies that $\ga(s)$ is a parameterization of a horizontal half-line starting from $0$. Since dilations preserve the horizontal distribution, $\ga(s)\in S_0$ for all $s\ge 0$.
\end{proof}

In the following we denote by $H$ the plane $t=0$.

\begin{lemma}\label{lem5e97064c}
	Let $S\subset \hh^1$ be a conical $C^1$ minimal $t$-graph.
	If $p\in S\setminus H$, then $p$ is a regular point and there are a singular point $q\in S\cap H$ and a horizontal half-line $L$ starting from $q$ and containing~$p$.	
\end{lemma}

\begin{proof}
	Let $p=(x,y,t)\in S$ with $t\neq0$. We know that the point $p$ is regular by Lemma~\ref{lem:half-line}.
	By Proposition~\ref{prop5e97045a}, there are $\hat v\in H$, with $|\hat v|=1$, and $s_0<0$ (possibly $s_0=-\infty$) such that 
	$\gamma((s_0,+\infty))\subset S$, where $\gamma(s) = p\exp(s\hat v)$, and $s_0$ is minimal with this property.
	We have two cases.\\
	First, if $s_0=-\infty$, then there is $s_1\in\R$ such that $\gamma(s_1)\in H\cap S$.
	Indeed, if this were not the case, the horizontal line $\gamma(\R)$ would meet the $t$-axis in a non-zero point contradicting the hypotheses that $S$ is a $t$-graph and $0\in S$.
	Now, notice that $\gamma'(s_1)=\hat v$ is not parallel to $\frac{\dd}{\dd \lambda}|_{\lambda=1}\delta_\lambda\gamma(s_1)$, but these two vectors are both horizontal and tangent to $S$. 
	Therefore $\gamma(s_1)$ is a singular point of $S$ and thus, the lemma is proven if we take $L=\gamma([s_1,+\infty))$ if $s_0>s_1$ or $L=\gamma((-\infty,s_1])$ if $s_0<s_1$.\\
	Second, if $s_0>-\infty$, then $\gamma(s_0)$ is a singular point and thus it belongs to $H$. 
	The lemma is proven if we take $L=\gamma([s_0,+\infty))$.
\end{proof}

\begin{lemma}\label{lem5e97064c2}
	Let $S\subset \hh^1$ be a $C^1$ minimal surface invariant by dilations centered at $0$.
	If $S_0=\{0\}$, then $S$ is the horizontal plane $\{t=0\}$.
\end{lemma}
\begin{proof}
	Let $p_0=(x_0,y_0,t_0)\in S$ and suppose that $t_0\neq0$.
	Since $p_0$ is a regular point but no horizontal line passing through $p_0$ contains $0$, then, by Proposition~\ref{prop5e97045a}, there is $\hat v\in H$ such that the entire line $s\mapsto p_0\exp(s\hat v)$ is contained in $S$.
	Since $S$ is a cone, for every $s\in\R$ and $\lambda>0$, we have
	\[
	\delta_\lambda(p_0\exp(s\hat v/\lambda)) 
	= (\lambda x_0, \lambda y_0, \lambda^2 t_0)\exp(s\hat v) \in S .
	\]
	Next, we claim that, if $t_0\neq0$, then $0$ is not the only singular point.
	Indeed, without loss of generality, we assume $\hat v= X$. Then the function $\phi(\lambda,s) := \delta_\lambda(p_0\exp(s\hat v/\lambda)) = (\lambda x_0+s ,\lambda y_0 , \lambda^2t_0 + \lambda s y_0)$ parametrizes a part of $S$.
	Notice that $\omega|_{\phi(s,\lambda)}[\de_\lambda\phi(\lambda,s)] = 2(\lambda t_0+sy_0)$ and $\omega|_{\phi(\lambda,s)}[\de_s\phi(\lambda,s)] = 0$.
	Therefore, $\phi\{(\lambda,s):\lambda t_0+sy_0=0\} \subset S_0$ and thus,
	if $t_0\neq0$, then $S_0\neq\{0\}$.
	
	However, the fact that $0$ is not the only singular point is in contradiction with the assumption $S_0=\{0\}$.
	Therefore, $t_0=0$ and so $S\subset H$.
	Since $S$ is a cone, $S=H$.
%	\footnotetext{Without loss of generality, we assume $\hat v= X$. Then the function $\phi(\lambda,s) := \delta_\lambda(p_0\exp(s\hat v/\lambda)) = (\lambda x_0+s ,\lambda y_0 , \lambda^2z_0 + \lambda s y_0)$ parametrizes a part of $S$.
%	Notice that $\omega|_{\phi(s,\lambda)}[\de_\lambda\phi(\lambda,s)] = 2(\lambda z_0+sy_0)$ and $\omega|_{\phi(\lambda,s)}[\de_s\phi(\lambda,s)] = 0$.
%	Therefore, $\phi\{(\lambda,s):\lambda z_0+sy_0=0\} \subset S_0$.}
\end{proof}

\begin{proof}[Proof of Theorem~\ref{thm5ee1fc47}] 
In case $S$ has no singular points, Lemma~\ref{lem:plane} implies that $S$ is a vertical plane.
If $S$ has only $0$ as a singular point, then $S=H$ by Lemma~\ref{lem5e97064c2}.

Finally let us assume that $S_0$ contains at least two points and that $S$ is invariant by dilations centered at $0$. Then $0$ is a singular point, and  Lemma~\ref{lem:half-line} implies that $S_0$ is a union of horizontal half-lines leaving the origin. 

Since $0$ is a singular point, $S$ can be represented near $0$ as the $t$-graph of a $C^1$ function and thus, since $S$ is a cone, the whole $S$ is the $t$-graph of a function $u:\R^2\to\R$.

If $L$ is one of the singular half-lines leaving the origin and $p\in L\setminus\{0\}$, then there is a neighborhood $U$ of $p$ such that $U\cap S_0=U\cap L$, because $T_pS=pH\neq H$ and $S_0\subset H$.
Therefore, these singular half-lines cannot accumulate 
and so we have a finite number of them $L_1,\ldots,L_k$, with $k\ge 1$
(see also Theorem C(b) in \cite{MR2983199}).

If $p\in S\setminus H$, then $p$ is a regular point and, by Lemma~\ref{lem5e97064c}, there is a half-line $L\subset S$ starting from a singular point $q\in S\cap H$, say $q\in L_j$.
Then $\bigcup_{\lambda>0}\delta_\lambda L$ describes $S$ on one side of $L_j$.
In other words, there is an arc $I_j^1$ so that $L_j$ is on the boundary of $0\#I_j^1$, so that $u$ is a version of the function $u_\alpha|_{\{y\ge0\}}$ or $u_\alpha|_{\{y\le0\}}$ in~\eqref{eq5e963809} on $0\#I_j^1$.

We conclude that for every $j\in\{1,\ldots,k\}$ there are two arcs $I_j^1$ and $I_j^2$, possibly empty, such that $S\setminus H$ is the graph of $u$ above 
$\bigcup_j (0\#I_j^{1}\cup 0\#I_j^{2})$.

Notice that $u$ does not have other singular points on $0\#\overline{I_j^1}$ other than $L_j$.
Therefore, if $i\neq j$, then $L_j\cap (0\#I_j^{1}\cup 0\#I_j^{2})=\emptyset$.
It follows that, in fact, $0\#I_j^{1}$ and $0\#I_j^{2}$ are never empty.
Indeed, on each side of every $L_j$ there are regular points, as we noticed above. 
However, they cannot belong to a sector of a singular half-line other than $L_j$.

Finally, the unit normal to $S$ is constant over each $0\#I_j^{1}\cup 0\#I_j^{2}$.
By Lemma~\ref{lem5e973be0}, it must have zero distributional divergence, while, 
by Proposition~\ref{prop5e95c6fe}, this happens exactly when it reflects across $L_j$, that is, the characteristic lines in $I_j^{1}$ and in $I_j^{2}$ meet $L_j$ with the same angle,
exactly as it happens with the function $u_\alpha$ in~\eqref{eq5e963809}.

Set $I_j=I_j^{1}\cup I_j^{2}$.
It is clear that $L_j$ is the bisectrix of $I_j$ and that $I_j\cap I_i=\emptyset$ 
if $i\neq j$.
\end{proof}

\bibliography{BIB-area-minimizing-cones}

\end{document}